\documentclass[11pt]{article}
\usepackage{amsfonts}
\usepackage{amsmath}
\usepackage{hyperref}
\usepackage{cite}
\usepackage{graphicx}
\usepackage{xtab}

\newtheorem{theorem}{Theorem}

\newtheorem{corollary}[theorem]{Corollary}

\newtheorem{definition}[theorem]{Definition}
\newtheorem{example}[theorem]{Example}

\newtheorem{notation}[theorem]{Notation}

\newtheorem{proposition}[theorem]{Proposition}
\newtheorem{remark}[theorem]{Remark}

\newenvironment{proof}[1][Proof]{\noindent \textbf{#1.} }{\  \rule{0.5em}{0.5em}}

\begin{document}

\title{Ruled surfaces and hyper-dual tangent sphere bundle}
\author{K. Derkaoui\thanks{%
Email{\small : }derkaouikhdidja248@hotmail.com}, \textsc{F. Hathout}\thanks{%
Email{\small : f.hathout@gmail.com \& fouzi.hathout@univ-saida.dz}},\textsc{%
\ M. Bekar}\thanks{%
Corresponding author's email: murat-bekar@hotmail.com}\textsc{\ and Y. Yayli}%
\thanks{%
Email: yayli@science.ankara.edu.tr} \\
$^{\ast }${\small Department of Mathematics, Chlef University, 02000 Chlef,
Algeria.}\\
$^{\dag }${\small Department of Computer science, Sa\"{\i}da University,
20000 Sa\"{\i}da, Algeria.}\\
$^{\ddag }${\small Department of Mathematics Education, Gazi University,
Ankara, Turkey}\\
$^{\S }${\small Department of Mathematics, Ankara University, Ankara, Turkey.%
}}
\date{}
\maketitle

\begin{abstract}
In this study, we define the unit hyper-dual sphere $\mathbb{S}_{\mathbb{D}%
_{2}}$ in hyper-dual vectors $\mathbb{D}_{2}$ and we give E-Study map
version in $\mathbb{D}_{2}$ which prove that $\mathbb{S}_{\mathbb{D}%
_{2}}^{2} $ is isomorphism to the tangent bundle $T\mathbb{S}_{\mathbb{D}%
}^{2}.$ Next, we define ruled surfaces in $\mathbb{D}$, we give its
developability condition and a geometric interpretation in $\mathbb{R}^{3}$
of any curves in $\mathbb{D}_{2}$. Finally, we present a relationship
between a ruled surfaces set in $\mathbb{R}^{3}$ and curves in hyper dual
vectors $\mathbb{D}_{2}$. We close each study with examples.\newline
\end{abstract}

\section{Introduction}

Hyper-dual numbers were introduced by Fike and Alonso in \cite{fa} which are
extended versions of dual numbers proposed in the 19th century by Cliford in
\cite{c}\textbf{. }Hyper-dual numbers are ones with a real component and a
number of infinitesimal components, usually written as $a_{0}+\varepsilon
a_{1}+\varepsilon ^{\ast }a_{2}+\varepsilon \varepsilon ^{\ast }a_{3}$ where
$\varepsilon $ and $\varepsilon ^{\ast }$ are two different numbers such
that $\varepsilon ^{2}=\varepsilon ^{\ast 2}=0$ and $\varepsilon
,\varepsilon ^{\ast }\neq 0$.

As dual number, the hyper-dual numbers find many applications. Firstly they
have been proposed in an approach to automatic differentiation and after, a
several works use hyper dual numbers in dynamics, complex software, analysis
to design airspace systems and open kinematic chain robot manipulator \cite%
{ch,ch1,fc,fa,rb,v}. Hyper-dual numbers has been also extended to
third-order and generalised to dual number of higher order see \cite{cc,ch2}%
, and from the algebra of hyper-dual vectors which have an interpretation in
geometry by E-study map and special surfaces in Euclidean space \cite%
{aby,hby,st}.

In order hand, ruled surfaces are traced out by the movement of a straight
line through space, and they are usually described by a correspondence
between parametric curves $\beta (t)$ and $\alpha (t).$ The parametric
description of a ruled surface is%
\begin{equation*}
\Phi \left( t,u\right) =\beta (t)+u\alpha (t)
\end{equation*}%
where $\beta $ is the directrix (or base curve) and $\alpha $ the director
curve. A many application of ruled surfaces can be found in computer-aided
geometric design (CAGD), surface design, manufacturing technology,
simulation of rigid bodies and dual vectors (see \cite{b,hby,it,o}).

The pure pose of this study is extended a geometric interpretation of
E-study map for a dual numbers to hyper-dual vectors and general case for
dual numbers of higher order, and present a correspondence between
hyper-dual vectors and ruled surfaces in Euclidean 3-space.

The paper is organised as follow; The section 2 is devoted to the needed
geometry tools of unit sphere in Euclidean $3$-space, a used notations and
an overview on dual vectors. In section 3, we present the algebra of
hyper-dual numbers and vectors denoted $D_{2}$ and $\mathbb{D}_{2},$
respectively. Next, we give the E-Study map version in $\mathbb{D}_{2}$
which can be extended to dual vectors of higher order $\mathbb{D}_{n}$ and
as consequence, we prove that the hyper-dual unit sphere $\mathbb{S}_{%
\mathbb{D}_{2}}^{2}$ is isomorphism to the tangent bundle of $T\mathbb{S}_{%
\mathbb{D}}^{2}$. In section 4, we define ruled surfaces in $\mathbb{D}_{2}$
with its developability condition and we give a geometric interpretation in $%
\mathbb{R}^{3}$ of any curves in $\mathbb{D}_{2}$. Finally, in section 5, we
give a relationship between a ruled surfaces set in $\mathbb{R}^{3}$ and
curves in hyper dual vectors $\mathbb{D}_{2}$. Each sections 4 and 5 are
closed with examples.

\section{Dual numbers and tangent bundle of unit sphere}

We firstly consider some needed definition of tangent bundle of sphere and
ruled surface. In Euclidean $3$-space $\mathbb{E}^{3}=\left( \mathbb{R}%
^{3},\cdot \right) $, the tangent bundle of the unit $2$-sphere $\mathbb{S}%
^{2}$ is given by%
\begin{equation}
T\mathbb{S}^{2}=\{(\gamma ,v)\in \mathbb{R}^{3}\times \mathbb{R}^{3}\mid
\left \vert \gamma \right \vert =1\  \text{and }\gamma \cdot v=0\}
\label{1.1}
\end{equation}%
and the unit tangent bundle is a hypersurface in $T\mathbb{S}^{2}$ defined as%
\begin{eqnarray}
UT\mathbb{S}^{2} &=&\{(\gamma ,v)\in \mathbb{R}^{3}\times \mathbb{R}^{3}\mid
\left \vert \gamma \right \vert =\left \vert v\right \vert =1\  \text{and }%
\gamma \cdot v=0\}  \label{1.2} \\
&=&\{(\gamma ,v)\in \mathbb{S}^{2}\times \mathbb{S}^{2}\mid \gamma \cdot
v=0\}  \notag
\end{eqnarray}%
which is a $3$-dimensional contact manifold$,$ where "$\cdot $" and "$\times
$" denotes, respectively, the usual inner and vector products in $\mathbb{R}%
^{3}$.

Now let $\Phi $ be a surface in $\mathbb{E}^{3}.$ $\Phi $ is a ruled surface
if its parametrization is in the form%
\begin{equation*}
\Phi (t,v)=\beta (t)+u\alpha (t),
\end{equation*}%
where $\beta (t)$ is the directrix (or base curve) and $\alpha (t)$ the
director curve. The surface $\Phi $ is said to be developable if%
\begin{equation*}
\det (\beta ^{\prime }(t),\alpha (t),\alpha ^{\prime }(t))=0
\end{equation*}%
In all sequel, we use the following a list of symbols and notations:

\begin{notation}
.\newline
$a:$ Real number (Tag normal)\newline
$\mathbf{a}:$ Real vector in $\mathbb{R}^{3}$ (Tag bold)\newline
$A:$ Dual number on $\mathbb{R}$ (Capital letter, tag normal)\newline
$D:$ Set of dual numbers on $\mathbb{R}$\newline
$\mathbb{A}:\ $Dual vector on $\mathbb{R}^{3}$(Capital letter,tag
blackboardbold)$\newline
\mathbb{D}:$ Set of dual vector on $\mathbb{R}^{3}$\newline
$A_{2}:$ Hyper-dual number on $D$ (Capital letter, tag normal indexed by 2)%
\newline
$D_{2}:$ Set of hyper-dual number on $\mathbb{R}$\newline
$\mathbb{A}_{2}$: Hyper-dual vector on $\mathbb{D}$(Capital letter, tag
blackboard bold indexed by 2)\newline
$\mathbb{D}_{2}:$ Set of hyper-dual vector on $\mathbb{R}^{3}$\newline
Lowercase greek, Uppercase greek, Uppercase greek indexed by $2$ characters
are curves in $D,\mathbb{D},\mathbb{D}_{2},$ respectively, except with a
note.
\end{notation}

\subsection{Dual numbers and the tangent bundle of unit sphere}

The set of dual numbers on $\mathbb{R}$ is defined to be%
\begin{equation*}
D\  \mathbb{=\{}A=a+\varepsilon a^{\ast }\mid a,a^{\ast }\in \mathbb{R},\text{
}\varepsilon \neq 0\text{ and }\varepsilon ^{2}=0\mathbb{\}}
\end{equation*}%
where $\varepsilon $ is the dual unit satisfying $r\varepsilon =\varepsilon
r $ for all $r\in \mathbb{R}$ and $a$, $a^{\ast }$ are called, respectively,
the non-dual and dual parts of $A$. If $\mathbf{a}$ and $\mathbf{a}^{\ast }$
are vectors in $\mathbb{R}^{3}$, then the set, denoted by $\mathbb{D},$ of
the the combination%
\begin{equation*}
\mathbb{A}=\mathbf{a}+\varepsilon \mathbf{a}^{\ast }
\end{equation*}%
is called the set of dual vector. The addition, inner product and vector
product on $\mathbb{D}$ of two dual vectors $\mathbb{A}=\mathbf{a}%
+\varepsilon \mathbf{a}^{\ast }$ and $\mathbb{B}=\mathbf{b}+\varepsilon
\mathbf{b}^{\ast }$ are defined as%
\begin{equation*}
\mathbb{A}+\mathbb{B=}(\mathbf{a}+\mathbf{b})+\varepsilon (\mathbf{a}^{\ast
}+\mathbf{b}^{\ast }),
\end{equation*}%
\begin{equation*}
\left \langle \mathbb{A},\mathbb{B}\right \rangle =\mathbf{a}\cdot \mathbf{b}%
+\varepsilon (\mathbf{a}^{\ast }\cdot \mathbf{b}+\mathbf{a}\cdot \mathbf{b}%
^{\ast })
\end{equation*}%
and%
\begin{equation*}
\mathbb{A}\times \mathbb{B=}\mathbf{a}\times \mathbf{b}+\varepsilon (\mathbf{%
a}\times \mathbf{b}^{\ast }+\mathbf{a}^{\ast }\times \mathbf{b})
\end{equation*}%
respectively. The norm of $\mathbb{A},$ exists only if $\mathbf{a}=0,$ is
defined to be%
\begin{equation}
\left \vert \mathbb{A}\right \vert =\sqrt{<\mathbb{A},\mathbb{A}>}=\left
\vert \mathbf{a}\right \vert +\varepsilon \frac{\mathbf{a}\cdot \mathbf{a}%
^{\ast }}{\left \vert \mathbf{a}\right \vert }\text{.}  \label{2.1}
\end{equation}%
(for a detail of dual numbers square root see \cite{aby,ch2,hby,v}). $%
\mathbb{A}$ is said to be a unit if $\left \vert \mathbb{A}\right \vert =1$
that means $\left \vert \mathbf{a}\right \vert =1$ and $<\mathbf{a},\mathbf{a%
}^{\ast }>=0$. Hence, we define the unit dual sphere by%
\begin{eqnarray}
\mathbb{S}_{\mathbb{D}}^{2} &=&\mathbb{\{A}=\mathbf{a}+\varepsilon \mathbf{a}%
^{\ast }\in \mathbb{D}\mid \left \vert \mathbb{A}\right \vert =1\mathbb{\}}
\label{2.2} \\
&=&\mathbb{\{(}\mathbf{a},\mathbf{a}^{\ast }\mathbb{)}\in \mathbb{R}^{3}%
\mathbb{\times R}^{3}\mid \left \vert \mathbf{a}\right \vert =1\text{ and }%
\mathbf{a}\cdot \mathbf{a}^{\ast }=0\mathbb{\}}\text{.}  \notag
\end{eqnarray}%
which is isomorphic to the tangent bundle of unit sphere $\mathbb{S}^{2}$
given in (Eq.\ref{1.1}) by the isomorphism%
\begin{equation}
\begin{array}{c}
\mathbb{S}_{\mathbb{D}}^{2}\rightarrow T\mathbb{S}^{2} \\
\mathbb{A}=\mathbf{a}+\varepsilon \mathbf{a}^{\ast }\mapsto \Gamma (t)=(%
\mathbf{a}(t),\mathbf{a}^{\ast }(t))%
\end{array}
\label{2.3}
\end{equation}%
and we have the following theorem which its geometric interpretation is in
\cite{uc}.

\begin{theorem}[\textsc{E. Study Map.}]
\label{TE1}The oriented lines in $\mathbb{R}^{3}$ are in one-to-one
corresponding with the points of the unit dual sphere $\mathbb{S}_{\mathbb{D}%
}^{2}$.
\end{theorem}

\section{Hyper-dual numbers and the tangent bundle of T$\mathbb{S}^{2}$}

A hyper-dual numbers set is defined as%
\begin{eqnarray*}
D_{2} &=&\left \{ A_{2}=\left( a_{0}+\varepsilon a_{1}\right) +\varepsilon
^{\ast }\left( a_{2}+\varepsilon a_{3}\right) \mid a_{\overline{0,3}}\in
\mathbb{R}\right \} \\
&=&\left \{ A_{2}=A+\varepsilon ^{\ast }A^{\ast }\mid A,A^{\ast }\in D\right
\}
\end{eqnarray*}%
where $\varepsilon ^{2}=\varepsilon ^{\ast 2}=\left( \varepsilon \varepsilon
^{\ast }\right) ^{2}=0$ and $\varepsilon \neq 0,\varepsilon ^{\ast }\neq
0,\varepsilon \neq \varepsilon ^{\ast },\varepsilon \varepsilon ^{\ast
}=\varepsilon \varepsilon ^{\ast }\neq 0.$ Addition and multiplication of
any two hyper-dual numbers
\begin{eqnarray*}
A_{2} &=&A+\varepsilon ^{\ast }A^{\ast }=a_{0}+\varepsilon a_{1}+\varepsilon
^{\ast }a_{2}+\varepsilon \varepsilon ^{\ast }a_{3}\text{ and} \\
B_{2} &=&B+\varepsilon ^{\ast }B^{\ast }=b_{0}+\varepsilon b_{1}+\varepsilon
^{\ast }b_{2}+\varepsilon \varepsilon ^{\ast }b_{3}
\end{eqnarray*}%
are defined, respectively, as%
\begin{eqnarray*}
A_{2}+B_{2} &=&\left( A+A^{\ast }\right) +\varepsilon ^{\ast }\left(
B+B^{\ast }\right) \\
&=&\left( a_{0}+b_{0}\right) +\varepsilon \left( a_{1}+b_{1}\right)
+\varepsilon ^{\ast }\left( a_{2}+b_{2}\right) +\varepsilon \varepsilon
^{\ast }\left( a_{3}+b_{3}\right)
\end{eqnarray*}%
and%
\begin{eqnarray*}
A_{2}B_{2} &=&AB+\varepsilon ^{\ast }\left( AB^{\ast }+A^{\ast }B\right) \\
&=&\left( a_{0}b_{0}\right) +\varepsilon \left( a_{0}b_{1}+a_{1}b_{0}\right)
+\varepsilon ^{\ast }\left( a_{0}b_{2}+a_{2}b_{0}\right) \\
&&+\varepsilon \varepsilon ^{\ast }\left(
a_{0}b_{3}+a_{1}b_{2}+a_{2}b_{1}+a_{3}b_{0}\right)
\end{eqnarray*}%
Now, we define the hyper-dual vectors set by%
\begin{equation*}
\mathbb{D}_{2}\mathbb{=}\left \{ \mathbb{A}_{2}=\mathbb{A}+\varepsilon
^{\ast }\mathbb{A}^{\ast }\mid \mathbb{A},\mathbb{A}^{\ast }\in \mathbb{D}%
\text{ and }\varepsilon ^{\ast 2}=0\right \}
\end{equation*}%
The inner product and vector product on $\mathbb{D}_{2}$ of two hyper-dual
vectors $\mathbb{A}_{2}=\mathbb{A}+\varepsilon \mathbb{A}^{\ast }$ and $%
\mathbb{B}_{2}=\mathbb{B}+\varepsilon ^{\ast }\mathbb{B}^{\ast }$ , where $%
\mathbb{A},\mathbb{A}^{\ast },\mathbb{B},\mathbb{B}^{\ast }\in \mathbb{D}$
are defined as%
\begin{eqnarray}
\left \langle \mathbb{A}_{2},\mathbb{B}_{2}\right \rangle _{2} &=&\left
\langle \mathbb{A}\mathbf{,}\mathbb{B}\right \rangle +\varepsilon ^{\ast
}\left( \left \langle \mathbb{A}\mathbf{,}\mathbb{B}^{\ast }\right \rangle
+\left \langle \mathbb{A}^{\ast }\mathbf{,}\mathbb{B}\right \rangle \right)
\label{2.41} \\
&=&\mathbf{a}_{0}\cdot \mathbf{b}_{0}+\varepsilon \left( \mathbf{a}_{0}\cdot
\mathbf{b}_{1}+\mathbf{a}_{1}\cdot \mathbf{b}_{0}\right) +\varepsilon ^{\ast
}\left( \mathbf{a}_{0}\cdot \mathbf{b}_{2}+\mathbf{a}_{2}\cdot \mathbf{b}%
_{0}\right)  \notag \\
&&+\varepsilon \varepsilon ^{\ast }\left( \mathbf{a}_{0}\cdot \mathbf{b}_{3}+%
\mathbf{a}_{1}\cdot \mathbf{b}_{2}+\mathbf{a}_{2}\cdot \mathbf{b}_{1}+%
\mathbf{a}_{3}\cdot \mathbf{b}_{0}\right)  \notag
\end{eqnarray}%
and%
\begin{eqnarray}
\mathbb{A}_{2}\times _{2}\mathbb{B}_{2} &=&\mathbb{A}\mathbf{\times }\mathbb{%
B}+\varepsilon ^{\ast }\left( \mathbb{A}\mathbf{\times }\mathbb{B}^{\ast }+%
\mathbb{A}^{\ast }\mathbf{\times }\mathbb{B}\right)  \label{2.42} \\
&=&\mathbf{a}_{0}\mathbf{\times b}_{0}+\varepsilon \left( \mathbf{a}_{0}%
\mathbf{\times b}_{1}+\mathbf{a}_{1}\mathbf{\times b}_{0}\right)
+\varepsilon ^{\ast }\left( \mathbf{a}_{0}\mathbf{\times b}_{2}+\mathbf{a}%
_{2}\mathbf{\times b}_{0}\right)  \notag \\
&&+\varepsilon \varepsilon ^{\ast }\left( \mathbf{a}_{0}\mathbf{\times b}%
_{3}+\mathbf{a}_{1}\mathbf{\times b}_{2}+\mathbf{a}_{2}\mathbf{\times b}_{1}+%
\mathbf{a}_{3}\mathbf{\times b}_{0}\right)  \notag
\end{eqnarray}%
respectively. The norm of hyper-dual number $\mathbb{A}_{2}$ is%
\begin{eqnarray}
\left \vert \mathbb{A}_{2}\right \vert &=&\sqrt{<\mathbb{A}_{2},\mathbb{A}%
_{2}>_{2}}=\left \vert \mathbb{A}\right \vert +\varepsilon ^{\ast }\frac{<%
\mathbb{A},\mathbb{A}^{\ast }>}{\left \vert \mathbb{A}\right \vert }
\label{2.4} \\
&=&\left \vert \mathbf{a}_{0}\right \vert +\varepsilon \frac{\mathbf{a}%
_{0}\cdot \mathbf{a}_{1}}{\left \vert \mathbf{a}_{0}\right \vert }%
+\varepsilon ^{\ast }\frac{\mathbf{a}_{0}\cdot \mathbf{a}_{2}}{\left \vert
\mathbf{a}_{0}\right \vert }  \notag \\
&&+\varepsilon \varepsilon ^{\ast }\left( \frac{\mathbf{a}_{0}\cdot \mathbf{a%
}_{3}}{\left \vert \mathbf{a}_{0}\right \vert }+\frac{\mathbf{a}_{1}\cdot
\mathbf{a}_{2}}{\left \vert \mathbf{a}_{0}\right \vert }-\frac{\mathbf{a}%
_{0}\cdot \mathbf{a}_{1}+\mathbf{a}_{0}\cdot \mathbf{a}_{2}}{\left \vert
\mathbf{a}_{0}\right \vert ^{3}}\right)  \notag
\end{eqnarray}%
where $\mathbf{a}_{0}\neq 0$ for the existence of $\mathbb{A}_{2}.$
Moreover, $\mathbb{A}_{2}$ is called unit if its norm is $1$. (here, the
used hyper-dual numbers square root can be found in the detail in \cite%
{aby,ch2,v})

Similar as unit dual sphere $\mathbb{S}_{\mathbb{D}}^{2}$ , we define the
hyper-dual unit sphere $\mathbb{S}_{\mathbb{D}_{2}}^{2}$ as%
\begin{eqnarray}
\mathbb{S}_{\mathbb{D}_{2}}^{2} &=&\mathbb{\{A}_{2}=\mathbb{A}+\varepsilon
^{\ast }\mathbb{A}^{\ast }\in \mathbb{D}_{2}\mid \left \vert \mathbb{A}%
_{2}\right \vert =1\mathbb{\}}  \label{2.5} \\
&=&\mathbb{\{(\mathbb{A}},\mathbb{\mathbb{A}^{\ast })}\in \mathbb{D}^{2}\mid
\left \vert \mathbb{\mathbb{A}}\right \vert =1\text{ and }<\mathbb{A},%
\mathbb{A}^{\ast }>=0\mathbb{\}}\text{.}  \notag
\end{eqnarray}%
and using the Eq.(\ref{2.4}), $\mathbb{S}_{\mathbb{D}}^{2}$ can be given as%
\begin{equation}
\mathbb{S}_{\mathbb{D}_{2}}^{2}=\left \{ \mathbb{(\mathbb{A}},\mathbb{%
\mathbb{A}^{\ast })}\in \mathbb{D}_{2}^{2}\mid \left \vert \mathbf{a}%
_{0}\right \vert =1\text{ and }\left \{
\begin{array}{c}
\mathbf{a}_{0}\cdot \mathbf{a}_{1}=\mathbf{a}_{0}\cdot \mathbf{a}_{2}=0 \\
\mathbf{a}_{0}\cdot \mathbf{a}_{3}=-\mathbf{a}_{1}\cdot \mathbf{a}_{2}%
\end{array}%
\right. \right \} \text{.}  \label{2.51}
\end{equation}

\begin{theorem}
\label{T1}The hyper-dual unit sphere $\mathbb{S}_{\mathbb{D}_{2}}^{2}$ is
isomorphism to $T\mathbb{S}_{\mathbb{D}}^{2}$ the tangent bundle of $\mathbb{%
S}_{\mathbb{D}}^{2}$.
\end{theorem}

\begin{proof}
If we consider the map%
\begin{equation}
\begin{array}{c}
\mathbb{S}_{\mathbb{D}_{2}}^{2}\rightarrow T\mathbb{S}_{\mathbb{D}}^{2} \\
\mathbb{A}_{2}=\mathbb{A}+\varepsilon ^{\ast }\mathbb{A}^{\ast }\mapsto
\Gamma _{2}(t)=(\mathbb{A}\left( t\right) ,\mathbb{A}^{\ast }\left( t\right)
)%
\end{array}
\label{2.6}
\end{equation}%
we immediately notice that its an isomorphism, where $\mathbb{A}\left(
t\right) $ and $\mathbb{A}^{\ast }\left( t\right) $ are a curves in $\mathbb{%
D}.$
\end{proof}

Now, we can extended E. Study map to hyper-dual numbers in the following.

\begin{theorem}[\textsc{E. Study Hyper-Map.}]
The oriented lines in $\mathbb{D}$ are in one-to-one corresponding with the
points of the unit hyper-dual sphere $\mathbb{S}_{\mathbb{D}}^{2}$.
\end{theorem}

As consequence, we have an extended formula of Eq.(\ref{2.3}) given by%
\begin{equation}
\mathbb{S}_{\mathbb{D}_{2}}^{2}\rightarrow T\mathbb{S}_{\mathbb{D}%
}^{2}\rightarrow TT\mathbb{S}^{2}  \label{1.4}
\end{equation}%
where $TT\mathbb{S}^{2}$ is the tangent bundle of $T\mathbb{S}^{2}.$

By taking only a unit hyper-dual part in $\mathbb{S}_{\mathbb{D}_{2}}^{2},$
we can get the following corollary.

\begin{corollary}
\label{C1}The unit hyper-dual unit sphere $U\mathbb{S}_{\mathbb{D}_{2}}^{2}$
defined by
\begin{equation*}
U\mathbb{S}_{\mathbb{D}_{2}}^{2}=\left \{ \mathbb{A}_{2}=\mathbb{A}%
+\varepsilon ^{\ast }\mathbb{A}^{\ast }\in \mathbb{S}_{\mathbb{D}%
_{2}}^{2}\mid \left \vert \mathbb{A}^{\ast }\right \vert =1\right \}
\end{equation*}%
is isomorphism to the unit tangent bundle of $\mathbb{S}_{\mathbb{D}}^{2}$,
i.e.%
\begin{equation}
\begin{array}{c}
U\mathbb{S}_{\mathbb{D}_{2}}^{2}\rightarrow T_{1}\mathbb{S}_{\mathbb{D}%
}^{2}\rightarrow T_{1}\left( T\mathbb{S}^{2}\right) \\
\mathbb{A}_{2}=\mathbb{A}+\varepsilon ^{\ast }\mathbb{A}^{\ast }\mapsto
\Gamma _{2}(t)=(\mathbb{A}\left( t\right) ,\mathbb{A}^{\ast }\left( t\right)
)%
\end{array}
\label{1.5}
\end{equation}%
where the hyper dual part $\mathbb{A}^{\ast }$ is unit $(\left \vert \mathbb{%
A}^{\ast }\right \vert =1).$
\end{corollary}

\bigskip

The extension of the Eq.(\ref{2.3}) and the Theorem \ref{T1} for dual and
hyper dual numbers, respectively, to the dual numbers of higher order $%
\mathbb{D}_{n}$ given in \cite{ch2} is give the following Remark.

\begin{remark}
We can naturally extended the Theorem's \ref{T1} result to the dual numbers
of order $n$ $\mathbb{D}_{n}$ as%
\begin{equation*}
\mathbb{S}_{\mathbb{D}_{n}}^{2}\simeq T\mathbb{S}_{\mathbb{D}%
_{n-1}}^{2}\simeq \underset{n\text{ }times}{\underbrace{T\cdots T}}\mathbb{S}%
^{2}
\end{equation*}%
where $\underset{n\text{ }times}{\underbrace{T\cdots T}}\mathbb{S}^{2}=%
\underset{n\text{ }times}{\underbrace{T\left( \cdots T\left( T\mathbb{S}%
^{2}\right) \right) }}$ and for \textsc{E. Study map} and \textsc{E. Study
Hyper-Map} to \textsc{E. Study }map of order $n$.
\end{remark}

\section{Ruled surfaces according to hyper-dual tangent sphere bundle}

We consider, in this section, the rules surface in $\mathbb{D}$ and its
developability condition and we interpret geometrically in $\mathbb{R}^{3}$
any curves in $\mathbb{D}$.

\begin{definition}
Ruled surface with parameter $t$ in $\mathbb{D}$ is defined by
\begin{equation*}
\Phi _{2}\left( t,\mathbb{U}\right) =\mathbb{B}\left( t\right) +\mathbb{UB}%
^{\ast }\left( t\right)
\end{equation*}%
where $\mathbb{B}\left( t\right) $ is the base curve, $\mathbb{B}^{\ast
}\left( t\right) $ is a ruling of $\Phi $ and $\mathbb{U}=u+\varepsilon
u^{\ast }\in D$.
\end{definition}

\begin{theorem}
\label{T2}Let denoted by $\Gamma _{2}(t)=\left( \mathbb{A}(t),\mathbb{A}%
^{\ast }(t)\right) $ the hyper-dual curve on hyper-dual unit sphere $\mathbb{%
S}_{\mathbb{D}_{2}}^{2}$, then each curve on $\mathbb{S}_{\mathbb{D}%
_{2}}^{2}(\cong T\mathbb{S}_{\mathbb{D}}^{2})$ corresponds to unique ruled
surface in $\mathbb{D}$ i.e.%
\begin{equation*}
\begin{array}{c}
\mathbb{S}_{\mathbb{D}_{2}}^{2}\rightarrow RS\left( \mathbb{D}\right) \\
\Gamma _{2}(t)\mapsto \Phi _{2}\left( t,\mathbb{U}\right) =\mathbb{A}\left(
t\right) \times \mathbb{A}^{\ast }\left( t\right) +\mathbb{UA}\left( t\right)%
\end{array}%
\end{equation*}%
where $RS\left( \mathbb{D}\right) $ is the set of the ruled surfaces in $%
\mathbb{D}$ and $\mathbb{U=}u+\varepsilon u^{\ast }$ is dual number.
\end{theorem}

The Theorem \ref{T2} generalise the particular case given by the Theorem 2
in \cite{hby}. Using the the Corollary \ref{C1}, the isomorphism%
\begin{equation*}
U\mathbb{S}_{\mathbb{D}_{2}}^{2}\simeq T_{1}\mathbb{S}_{\mathbb{D}%
}^{2}=\left \{ \mathbb{A}_{2}=\mathbb{A}+\varepsilon ^{\ast }\mathbb{A}%
^{\ast }\mid \left \vert \mathbb{A}\right \vert =\left \vert \mathbb{A}%
^{\ast }\right \vert =1,\  \left \langle \mathbb{A},\mathbb{A}^{\ast }\right
\rangle =0\right \}
\end{equation*}%
and the Eq.(\ref{2.51}), the last equation turns to%
\begin{equation}
U\mathbb{S}_{\mathbb{D}_{2}}^{2}=\left \{ \mathbb{A}_{2}\in \mathbb{D}%
_{2}\mid \left \vert \mathbf{a}_{0}\right \vert =\left \vert \mathbf{a}%
_{2}\right \vert =1\text{ and }\left \{
\begin{array}{c}
\mathbf{a}_{0}\cdot \mathbf{a}_{1}=\mathbf{a}_{0}\cdot \mathbf{a}_{2}=%
\mathbf{a}_{2}\cdot \mathbf{a}_{3}=0 \\
\mathbf{a}_{0}\cdot \mathbf{a}_{3}=-\mathbf{a}_{1}\cdot \mathbf{a}_{2}%
\end{array}%
\right. \right \}  \label{3.1}
\end{equation}%
where $\mathbb{A}=\mathbf{a}_{0}+\varepsilon \mathbf{a}_{1}$ and $\mathbb{A}%
^{\ast }=\mathbf{a}_{2}+\varepsilon \mathbf{a}_{3}.$ We have the following
theorem.

\begin{proposition}
\label{P1}The corresponding ruled surface of hyper-dual curve $\Gamma
_{2}(t)=\left( \mathbb{A}(t),\mathbb{A}^{\ast }(t)\right) ,$ in hyper-dual
unit sphere $\mathbb{S}_{\mathbb{D}_{2}}^{2},$ is developable in $\mathbb{D}$%
\ if and only if
\begin{equation*}
\mathbf{a}_{0}^{\prime }\cdot \mathbf{a}_{2}^{\prime }=0\text{ and }\mathbf{a%
}_{0}^{\prime }\cdot \mathbf{a}_{3}^{\prime }=-\mathbf{a}_{1}^{\prime }\cdot
\mathbf{a}_{2}^{\prime }
\end{equation*}
\end{proposition}

\begin{proof}
Let $\Phi _{2}\left( t,\mathbb{U}\right) =\mathbb{A}\left( t\right) \times
\mathbb{A}^{\ast }\left( t\right) +\mathbb{UA}\left( t\right) $ be a
corresponding ruled surface of hyper-dual curve $\Gamma _{2}(t)=\left(
\mathbb{A}(t),\mathbb{A}^{\ast }(t)\right) $ in $\mathbb{S}_{\mathbb{D}%
_{2}}^{2}.$ We have%
\begin{equation*}
\det \left( \left( \mathbb{A}\times \mathbb{A}^{\ast }\right) ^{\prime },%
\mathbb{A}^{\prime },\mathbb{A}\right) =\left \langle \mathbb{A}^{\prime },%
\mathbb{A}^{\ast \prime }\right \rangle
\end{equation*}%
Using the developability condition and the Eq.(\ref{2.41}), the last
equation turns to%
\begin{equation*}
\mathbf{a}_{0}^{\prime }\times \mathbf{a}_{2}^{\prime }+\varepsilon (\mathbf{%
a}_{0}^{\prime }\times \mathbf{a}_{3}^{\prime }+\mathbf{a}_{1}^{\prime
}\times \mathbf{a}_{2}^{\prime })=0
\end{equation*}%
and%
\begin{equation*}
\mathbf{a}_{0}^{\prime }\cdot \mathbf{a}_{2}^{\prime }=0\text{ and }\mathbf{a%
}_{0}^{\prime }\cdot \mathbf{a}_{3}^{\prime }=-\mathbf{a}_{1}^{\prime }\cdot
\mathbf{a}_{2}^{\prime }
\end{equation*}%
which proof the Proposition.

\begin{corollary}
\label{C11}Let $\Gamma _{2}(t)=\left( \mathbb{A}(t),\mathbb{A}^{\ast
}(t)\right) $ be a hyper-dual curve in $\mathbb{S}_{\mathbb{D}_{2}}^{2}$.
The corresponding ruled surface $\Phi _{2}\left( t,\mathbb{U}\right) $ of $%
\Gamma _{2}$ is developable if and only if \newline
\textbf{i.} $\Gamma _{2}$ and $\mathbb{A}(t)$ have the same arc-length
parameter or,\newline
\textbf{ii}. $\Gamma _{2}$ and $\mathbf{a}_{0}(t)$ have the same arc-length
parameter and the corresponding ruled surface $\Phi \left( t,u\right) $ in $%
\mathbb{R}^{3}$ to $\mathbb{A}$ is developable.
\end{corollary}
\end{proof}

\begin{proof}
The norm of $\Gamma _{2}^{\prime }$ is
\begin{equation*}
\left \vert \Gamma _{2}^{\prime }(t)\right \vert =\left \vert \mathbb{A}%
^{\prime }\right \vert +\varepsilon ^{\ast }\frac{<\mathbb{A}^{\prime },%
\mathbb{A}^{\prime \ast }>}{\left \vert \mathbb{A}^{\prime }\right \vert }
\end{equation*}%
Its arc-length parameter $t_{2}$ is%
\begin{eqnarray}
t_{2} &=&\int_{0}^{t}\left \vert \Gamma _{2}^{\prime }(s)\right \vert ds
\notag \\
&=&\int_{0}^{t}\left \vert \mathbb{A}^{\prime }(s)\right \vert
ds+\varepsilon ^{\ast }\int_{0}^{t}\frac{<\mathbb{A}^{\prime }(s),\mathbb{A}%
^{\prime \ast }(s)>}{\left \vert \mathbb{A}^{\prime }(s)\right \vert }ds
\label{3.3} \\
&=&\int_{0}^{t}\left \vert \mathbf{a}_{0}^{\prime }(s)\right \vert
ds+\varepsilon \int_{0}^{t}\frac{<\mathbf{a}_{0}^{\prime }(s),\mathbf{a}%
_{1}^{\prime }(s)>}{\sqrt{<\mathbf{a}_{0}^{\prime }(s),\mathbf{a}%
_{0}^{\prime }(s)>}}ds+\varepsilon ^{\ast }\int_{0}^{t}\frac{<\mathbb{A}%
^{\prime }(s),\mathbb{A}^{\prime \ast }(s)>}{\left \vert \mathbb{A}^{\prime
}(s)\right \vert }ds  \notag \\
&=&t+\varepsilon \int_{0}^{t}\frac{<\mathbf{a}_{0}^{\prime }(s),\mathbf{a}%
_{1}^{\prime }(s)>}{\sqrt{<\mathbf{a}_{0}^{\prime }(s),\mathbf{a}%
_{0}^{\prime }(s)>}}ds+\varepsilon ^{\ast }\int_{0}^{t}\frac{<\mathbb{A}%
^{\prime }(s),\mathbb{A}^{\prime \ast }(s)>}{\left \vert \mathbb{A}^{\prime
}(s)\right \vert }ds  \label{3.5}
\end{eqnarray}%
Using the Proposition \ref{P1} the Eq.(\ref{3.3}), we obtaint the proof of
the assertion (\textit{i}). From the Proposition $4$ page $6$ in \cite{hby}
and the Eq.(\ref{3.5}) we can easily get the assertion (\textit{ii}).
\end{proof}

Note that the Corollary \ref{C11} present an extanded result of Proposition $%
4$ given in \cite{hby} to hyper-dual context.

\begin{proposition}
To each curve on $\mathbb{S}_{\mathbb{D}_{2}}^{2}$ corresponds a ruled
surfaces in $\mathbb{R}^{3}$ and a congruance ruled surfaces with a common
base curves and orthogonal ruling directions in $\mathbb{R}^{3}$.
\end{proposition}

\begin{proof}
Let $\Gamma _{2}\left( t\right) =(\mathbb{A}\left( t\right) ,\mathbb{A}%
^{\ast }\left( t\right) )$ be a curve on $\mathbb{S}_{\mathbb{D}_{2}}^{2},$
by the Theorem \ref{T2}, we have the corresponding ruled surface%
\begin{eqnarray*}
\Phi _{2}\left( t,\mathbb{U}\right) &=&\mathbb{A}\times \mathbb{A}^{\ast }+%
\mathbb{UA} \\
&=&\left( \mathbf{a}_{0}+\varepsilon \mathbf{a}_{1}\right) \times \left(
\mathbf{a}_{2}+\varepsilon \mathbf{a}_{3}\right) +\left( \mathbf{u}%
+\varepsilon \mathbf{u}^{\ast }\right) \left( \mathbf{a}_{0}+\varepsilon
\mathbf{a}_{1}\right) \\
&=&\mathbf{a}_{0}\times \mathbf{a}_{2}+\varepsilon \mathbf{a}_{1}\times
\mathbf{a}_{2}+\varepsilon \mathbf{a}_{0}\times \mathbf{a}_{3}+\mathbf{ua}%
_{0}+\mathbf{u}\varepsilon \mathbf{a}_{1}+\varepsilon \mathbf{u}^{\ast }%
\mathbf{a}_{0} \\
&=&\underset{I}{\underbrace{\mathbf{a}_{0}\times \mathbf{a}_{2}+\mathbf{ua}%
_{0}}}+\underset{II}{\varepsilon \left( \underbrace{\left( \mathbf{a}%
_{1}\times \mathbf{a}_{2}+\mathbf{a}_{0}\times \mathbf{a}_{3}\right) +%
\mathbf{ua}_{1}+\mathbf{u}^{\ast }\mathbf{a}_{0}}\right) }
\end{eqnarray*}%
The first part present a real dual part of ruled surface on $\mathbb{R}^{3}$%
, $\Phi \left( t,\mathbf{u}\right) =\mathbf{a}_{0}\times \mathbf{a}_{2}+%
\mathbf{ua}_{0}$ and the second part present a congruance ( or a family of)
ruled surfaces with common base curve
\begin{equation*}
K\left( t\right) =\left( \mathbf{a}_{1}\times \mathbf{a}_{2}\right) \left(
t\right) +\left( \mathbf{a}_{0}\times \mathbf{a}_{3}\right) \left( t\right)
\end{equation*}%
and perpendicular direction (i.e. orthogonal ruling $\mathbf{a}_{1}\cdot
\mathbf{a}_{0}=0$)$.$
\end{proof}

\begin{proposition}
To each curve on $\mathbb{S}_{\mathbb{D}_{2}}^{2}$ corresponds an infinite
ruled surface couple a ruled surfaces in $\mathbb{R}^{3}$ with a common base
curves and orthogonal ruling directions in $\mathbb{R}^{3}$.
\end{proposition}

\begin{proof}
Let $\Gamma _{2}\left( t\right) =(\mathbb{A}\left( t\right) ,\mathbb{A}%
^{\ast }\left( t\right) )$ be a curve on $\mathbb{S}_{\mathbb{D}_{2}}^{2},$
form the Proposition \ref{T2}, we have
\begin{eqnarray*}
\Phi _{2}\left( t,\mathbb{U}\right) &=&\mathbb{A}\times \mathbb{A}^{\ast }+%
\mathbb{UA} \\
&=&\mathbf{a}_{0}\times \mathbf{a}_{2}+\mathbf{ua}_{0}+\varepsilon \left(
\left( \mathbf{a}_{1}\times \mathbf{a}_{2}+\mathbf{a}_{0}\times \mathbf{a}%
_{3}\right) +\mathbf{ua}_{1}+\mathbf{u}^{\ast }\mathbf{a}_{0}\right)
\end{eqnarray*}%
where
\begin{equation*}
\left \{
\begin{array}{l}
I\left( t,u\right) =\mathbf{a}_{0}\times \mathbf{a}_{2}+\mathbf{ua}_{0} \\
II\left( t,u,u^{\ast }\right) =\left( \mathbf{a}_{1}\times \mathbf{a}_{2}+%
\mathbf{a}_{0}\times \mathbf{a}_{3}\right) +\mathbf{ua}_{1}+\mathbf{u}^{\ast
}\mathbf{a}_{0}%
\end{array}%
\right.
\end{equation*}%
present a ruled srface and a congruence ruled surface, respectively. Now,
suppose $\mathbf{u}^{\ast }=\mathbf{u}^{\ast }\left( t\right) $ is function,
let the couple ruled surface
\begin{equation}
\left \{
\begin{array}{l}
\overline{I}\left( t,u\right) =\mathbf{a}_{0}\times \mathbf{a}_{2}+f\mathbf{a%
}_{0}+\mathbf{ua}_{0} \\
\overline{II}\left( t,u,u^{\ast }\right) =\left( \mathbf{a}_{1}\times
\mathbf{a}_{2}+\mathbf{a}_{0}\times \mathbf{a}_{3}\right) +\mathbf{ua}_{1}+g%
\mathbf{a}_{1}+\mathbf{u}^{\ast }\left( t\right) \mathbf{a}_{0}%
\end{array}%
\right.  \label{2.61}
\end{equation}%
where $f$ and $g$ are function in parameter $t$. By taking the the functions
\begin{equation*}
\left \{
\begin{array}{c}
f\left( t\right) =\left \langle \mathbf{a}_{1}\times \mathbf{a}_{2},\mathbf{a%
}_{0}\right \rangle +\mathbf{u}^{\ast }\left( t\right) \\
g\left( t\right) =\left \langle \mathbf{a}_{0}\times \left( \mathbf{a}_{2}-%
\mathbf{a}_{3}\right) ,\mathbf{a}_{1}\right \rangle%
\end{array}%
\right.
\end{equation*}%
and substuting
\end{proof}

\section{Unit tangent bundle of $T\mathbb{S}^{2}$ and ruled surface}

In this section, we give a relationship between the choice of two ruled
surfaces in $\mathbb{R}^{3}$ and curves in hyper-dual numbers $\mathbb{D}%
_{2} $.

\begin{proposition}
To each curve on $U\mathbb{S}_{\mathbb{D}_{2}}^{2}$ corresponds two ruled
surfaces in $\mathbb{R}^{3}$ having the same base curve and their rulings
are perpendicular.
\end{proposition}

\begin{proof}
Let $\Gamma _{2}\left( t\right) $ be a curve on $U\mathbb{S}_{\mathbb{D}%
_{2}}^{2}$ then
\begin{equation*}
\Gamma _{2}\left( t\right) =\mathbb{A}\left( t\right) +\varepsilon ^{\ast }%
\mathbb{A}^{\ast }\left( t\right) \mid \left \vert \mathbb{A}\right \vert
=\left \vert \mathbb{A}^{\ast }\right \vert =1,\  \left \langle \mathbb{A},%
\mathbb{A}^{\ast }\right \rangle =0
\end{equation*}%
From the Theorem \ref{T2} and the Eq.(\ref{3.1}), $\Gamma _{2}$ correspond
to two ruled surfaces in $\mathbb{R}^{3}$ given by%
\begin{equation*}
\left \{
\begin{array}{c}
\Phi _{1}\left( t,u\right) =\left( \mathbf{a}_{0}\times \mathbf{a}%
_{1}\right) \left( t\right) +u\mathbf{a}_{0}\left( t\right) \\
\Phi _{2}\left( t,v\right) =\left( \mathbf{a}_{2}\times \mathbf{a}%
_{3}\right) \left( t\right) +v\mathbf{a}_{2}\left( t\right)%
\end{array}%
\right.
\end{equation*}%
where $\mathbb{A}=\mathbf{a}_{0}+\varepsilon \mathbf{a}_{1}$ and $\mathbb{A}%
^{\ast }=\mathbf{a}_{2}+\varepsilon \mathbf{a}_{3}.\ $However this two forms
did not give the Theorem's result, for this, for a new form of ruled surfaces%
\begin{equation*}
\left \{
\begin{array}{c}
\Phi _{1}\left( t\right) =\left( \mathbf{a}_{0}\times \mathbf{a}_{1}\right)
\left( t\right) +f\mathbf{a}_{0}\left( t\right) +u\mathbf{a}_{0}\left(
t\right) \\
\Phi _{2}\left( t\right) =\left( \mathbf{a}_{2}\times \mathbf{a}_{3}\right)
\left( t\right) +g\mathbf{a}_{2}\left( t\right) +u\mathbf{a}_{2}\left(
t\right)%
\end{array}%
\right.
\end{equation*}%
we get a same base curve $k(t)$ for the surfaces
\begin{equation*}
k(t)=\left( \mathbf{a}_{0}\times \mathbf{a}_{1}\right) \left( t\right)
+f\left( t\right) \mathbf{a}_{0}\left( t\right) =\left( \mathbf{a}_{2}\times
\mathbf{a}_{3}\right) \left( t\right) +g\left( t\right) \mathbf{a}_{2}\left(
t\right)
\end{equation*}%
where%
\begin{equation*}
\left \{
\begin{array}{c}
f=\mathbf{a}_{2}\times \mathbf{a}_{3}\cdot \mathbf{a}_{0} \\
g=\mathbf{a}_{0}\times \mathbf{a}_{1}\cdot \mathbf{a}_{2}%
\end{array}%
\right.
\end{equation*}%
and using the Eq.(\ref{3.1}), we easily get the orthogonality of ruling
curves i.e. $\left \langle \mathbf{a}_{0},\mathbf{a}_{2}\right \rangle =0.$
\end{proof}

\begin{example}
For a given unit curve $\alpha :I\subset \mathbb{R}\rightarrow \mathbb{R}%
^{3} $ with a Frent frame $\left \{ \mathbf{t},\mathbf{n},\mathbf{b}%
\right
\} ,$ let $\mathbb{A}$ and $\mathbb{A}^{\ast }$ be a dual vectors
given as%
\begin{equation*}
\mathbb{A}\left( t\right) =\mathbf{t}\left( t\right) +\varepsilon \mathbf{n}%
\left( t\right) \text{ and }\mathbb{A}^{\ast }=\mathbf{b}\left( t\right)
+\varepsilon \mathbf{n}\left( t\right)
\end{equation*}%
and hyper-dual vector $\Gamma _{2}\left( t\right) =\mathbb{A}\left( t\right)
+\varepsilon ^{\ast }\mathbb{A}^{\ast }\left( t\right) $ in $U\mathbb{S}_{%
\mathbb{D}_{2}}^{2}.$ The base curve is
\begin{eqnarray*}
k(t) &=&\mathbf{t}\times \mathbf{n+\det }\left( \mathbf{b},\mathbf{n,t}%
\right) \mathbf{t} \\
&=&\mathbf{b-t}
\end{eqnarray*}%
and the corresponding corresponds ruled surfaces in $\mathbb{R}^{3}$ are%
\begin{equation*}
\left \{
\begin{array}{c}
\Phi _{1}\left( t,u\right) =\left( \mathbf{b-t}\right) \left( t\right) +u%
\mathbf{t}\left( t\right) \\
\Phi _{2}\left( t,v\right) =\left( \mathbf{b-t}\right) \left( t\right) +v%
\mathbf{b}\left( t\right)%
\end{array}%
\right.
\end{equation*}%
where their rulings are perpendicular. The developable ruled surfaces $\Phi
_{1}$ and $\Phi _{2}$ along $K$ are normal approximation surfaces to each
other (see for normal approximation surfaces in \cite{ky}, \cite{s}).
\end{example}

\begin{example}
Let $\alpha :I\subset \mathbb{R}\rightarrow \mathbb{R}^{3}$ a curve with a
adapted frame apparatus $\left \{ \mathbf{n},\mathbf{c},\mathbf{w}\right \} $
where $\mathbf{n}$ is the normal vector of $\alpha ,$ $\mathbf{c=}\frac{%
\mathbf{n}^{\prime }}{^{\left \vert \mathbf{n}^{\prime }\right \vert }}$ and
$\mathbf{w=n}\times \mathbf{c}$ is the unit Darboux vector (for a detail see
\cite{bgy}). Let $\mathbb{A}\left( t\right) $ and $\mathbb{A}^{\ast }\left(
t\right) $ be a dual vectors given as%
\begin{equation*}
\mathbb{A}\left( t\right) =\mathbf{n}\left( t\right) +\varepsilon \mathbf{c}%
\left( t\right) \text{ and }\mathbb{A}^{\ast }=\mathbf{w}\left( t\right)
+\varepsilon \mathbf{c}\left( t\right)
\end{equation*}%
and hyper-dual curve $\Gamma _{2}\left( t\right) =\mathbb{A}\left( t\right)
+\varepsilon ^{\ast }\mathbb{A}^{\ast }\left( t\right) $ in $U\mathbb{S}_{%
\mathbb{D}_{2}}^{2}.$ The base curve is
\begin{eqnarray*}
k(t) &=&\mathbf{n}\times \mathbf{c+\det }\left( \mathbf{w},\mathbf{c,n}%
\right) \mathbf{n} \\
&=&\mathbf{w-n}
\end{eqnarray*}%
The corresponding corresponds ruled surfaces in $\mathbb{R}^{3}$ are%
\begin{equation*}
\left \{
\begin{array}{c}
\Phi _{1}\left( t,u\right) =\left( \mathbf{w-n}\right) \left( t\right) +u%
\mathbf{n}\left( t\right) \\
\Phi _{2}\left( t,v\right) =\left( \mathbf{w-n}\right) \left( t\right) +v%
\mathbf{w}\left( t\right)%
\end{array}%
\right.
\end{equation*}%
where their rulings $\mathbf{n}$ and $\mathbf{w}$\ are perpendicular.
Moreover, if $\alpha $ is slant helix then $\Phi _{1}$ and $\Phi _{2}$ are
developable constant angle surfaces. The surfaces $\Phi _{1}$ and $\Phi _{2}$
are normal approximation surfaces to each other.
\end{example}

Inversely, we have the following proposition.

\begin{proposition}
Let given two ruled surfaces in $\mathbb{R}^{3}$ with a same base curves,
then there exist two corresponding curves in $\mathbb{D}_{2}.$
\end{proposition}

\begin{proof}
Let $\Phi _{1,2}$ be two ruled surfaces in $\mathbb{R}^{3}$ with a same base
curve $K\left( t\right) $, given as%
\begin{equation*}
\left \{
\begin{array}{c}
\Phi _{1}\left( t,u\right) =K\left( t\right) +u\mathbf{a} \\
\Phi _{2}\left( t,v\right) =K\left( t\right) +v\mathbf{a}^{\ast }%
\end{array}%
\right.
\end{equation*}%
according the Theorem \ref{T2}, the corresponding dual vectors are%
\begin{equation*}
\left \{
\begin{array}{c}
\mathbb{A}=\mathbf{a}+\varepsilon \left( K\left( t\right) \times \mathbf{a}%
\right) \\
\mathbb{A}^{\ast }=\mathbf{a}^{\ast }+\varepsilon \left( K\left( t\right)
\times \mathbf{a}^{\ast }\right)%
\end{array}%
\right.
\end{equation*}%
where $\mathbb{A},\mathbb{A}^{\ast }$ are unit in $\mathbb{D},$ and which by
the Eq.(\ref{2.6}), give the the corresponding curves%
\begin{equation*}
\left \{
\begin{array}{c}
\Gamma _{1}\left( t\right) =\left( \mathbb{A}\left( t\right) ,\mathbb{A}%
^{\ast }\left( t\right) \right) \\
\Gamma _{2}\left( t\right) =\left( \mathbb{A}^{\ast }\left( t\right) ,%
\mathbb{A}\left( t\right) \right)%
\end{array}%
\right.
\end{equation*}%
in $\mathbb{D}_{2}.$
\end{proof}

\begin{example}
Let $K$ be a helix curve defined as
\begin{equation*}
K(t)=(r\cos t,r\sin t,ct).
\end{equation*}%
where $r$ and $c$ are non null real. We have unit speed and the unit normal
vectors%
\begin{eqnarray*}
\mathbf{t}_{K} &=&\frac{1}{\sqrt{r^{2}+c^{2}}}(-r\sin t,r\cos t,c) \\
\mathbf{n}_{K} &=&(-\cos t,-\sin t,0).
\end{eqnarray*}%
We gets a two parametric representation of the developable surfaces%
\begin{equation*}
\left \{
\begin{array}{l}
\Phi _{1}\left( t,u\right) =K\left( t\right) +u\mathbf{t}_{K}=\left( r\cos t-%
\frac{ur\sin t}{\sqrt{r^{2}+c^{2}}},r\sin t+\frac{ur\cos t}{\sqrt{r^{2}+c^{2}%
}},ct+\frac{uc}{\sqrt{r^{2}+c^{2}}}\right) \\
\Phi _{2}\left( t,v\right) =K\left( t\right) +v\mathbf{n}_{K}=\left( \left(
r-v\right) \cos t,\left( r-v\right) \sin t,ct\right)%
\end{array}%
\right.
\end{equation*}%
and the corresponding dual vectors are%
\begin{equation*}
\left \{
\begin{array}{l}
\mathbb{A}=\frac{1}{r^{2}+c^{2}}\left[ (-r\sin t,r\cos t,c)+\varepsilon
\left( cr\left( \sin t-t\cos t\right) ,cr\left( \cos t+t\sin t\right)
,r^{2}\right) \right] \\
\mathbb{A}^{\ast }=(-\cos t,-\sin t,0)+\varepsilon ct\left( \sin t,-\cos
t,0\right)%
\end{array}%
\right.
\end{equation*}%
which give a rise of two curves in $\mathbb{D}_{2}$ defined as%
\begin{equation*}
\Gamma _{1}\left( t\right) =\mathbb{A}\left( t\right) +\varepsilon ^{\ast }%
\mathbb{A}^{\ast }\left( t\right) \text{ and }\Gamma _{2}\left( t\right) =%
\mathbb{A}^{\ast }\left( t\right) +\varepsilon ^{\ast }\mathbb{A}\left(
t\right) .
\end{equation*}%
We present the two ruled surfaces $\Phi _{1,2}$ for $r=c=1$, in following
Figure \ref{f}.
\end{example}

\begin{figure}[h]
\centering
\includegraphics[width=2.7in,height=2.1in]{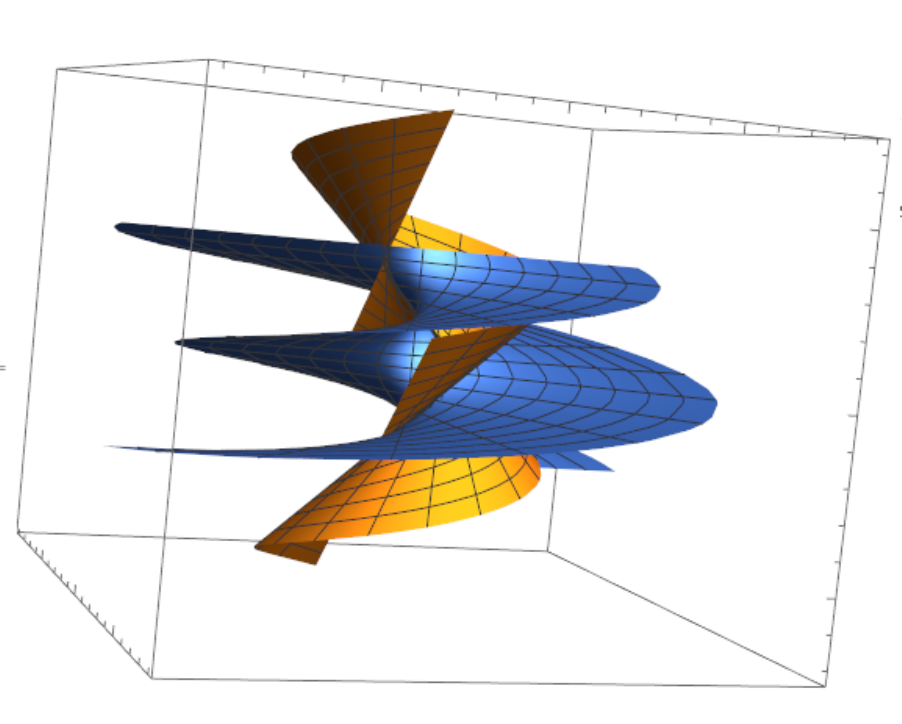}
\caption{Ruled surfaces $\Phi _{1,2}$ in
blue and yellow color, respectively.}
\label{f}
\end{figure}

\end{document}